\documentclass{amsart}
\usepackage{latexsym,amssymb,amsmath,amscd,amsxtra,amsfonts,amsthm}
\usepackage{fancyhdr,array,dsfont}
\usepackage{enumerate}

\makeatletter
\@namedef{subjclassname@2010}{%
  \textup{2010} Mathematics Subject Classification}
\makeatother

\numberwithin{equation}{section}

\newtheorem{theorem}{Theorem}[section]
\newtheorem{lemma}[theorem]{Lemma}
\newtheorem{proposition}[theorem]{Proposition}
\newtheorem{remark}[theorem]{Remark}

\newcommand{\beq}{\begin{eqnarray*}}
\newcommand{\eeq}{\end{eqnarray*}}
\newcommand{\beqn}{\begin{eqnarray}}
\newcommand{\eeqn}{\end{eqnarray}}

\begin{document}





\title[Ergodicity and Asymptotic Stability]{Ergodicity and Asymptotic Stability of Feller Semigroups on Polish Metric Spaces}

\author[F.-Z. GONG]{Fu-Zhou GONG}
\address{Fu-Zhou GONG, Institute of Applied Mathematics, Academy of Mathematics and Systems Science,
Chinese Academy of Sciences, Beijing 100190, China}
\email{fzgong@amt.ac.cn}

\author[Y. LIU]{Yuan LIU}
\address{Yuan LIU, Institute of Applied Mathematics, Academy of Mathematics and Systems Science,
Chinese Academy of Sciences, Beijing 100190, China}
\email{liuyuan@amss.ac.cn}

\date{\today}

\begin{abstract}
We provide some sharp criteria for studying the ergodicity and asymptotic stability of general Feller semigroups on Polish metric spaces. As application, the 2D Navier-Stokes equations with degenerate stochastic forcing will be simply revisited.
\end{abstract}



\maketitle

\allowdisplaybreaks

\section{Introduction}
 \label{Intro}

There is a vast literature of studying the ergodicity and asymptotic stability for various semigroups from dynamical systems and Markov chains. Abundant theories and applications have been established for compact or locally compact state spaces. However, it is difficult to extend them freely to infinite dimensional normed spaces or more general Polish metric spaces.

Actually, in the field of stochastic partial differential equations (SPDE for short), the uniqueness of ergodic measure can be derived from the {\em strong Feller} property besides {\em topological irreducibility}, which has been a routine to deal with the equations with non-degenerate additive noise. Recently, the {\em asymptotic strong Feller} property, as a celebrating breakthrough, was presented by Hairer and Mattingly \cite{Hairer} to solve the unique ergodicity for 2D Navier-Stokes equations with degenerate additive noise. They obtained such kind of property due to the gradient estimate (for some $\delta > 0$)
   \beqn
      |\nabla P_t\varphi(x)| \leqslant C(x) \left(||\varphi||_{\infty} + e^{-\delta t}||\nabla\varphi||_{\infty} \right), \ \ \forall t\geqslant t_0. \label{eqASFGradient}
   \eeqn
Similar arguments worked for semilinear SPDEs too, see Hairer and Mattingly \cite{Hairer-Semi}. However, sometimes the asymptotic strong Feller property is not easy to be verified (or even fails). For example, R\"{o}ckner, Zhu and Zhu \cite{RocknerZhuZhu} encountered such an obstacle in the study of stochastic quasi-geostrophic equations.

There were also other notable contributions to this subject, which came from Lasota and Szarek \cite{Lasota-Szarek} and their subsequent works for {\em equicontinuous} semigroups, or called {\em e-chains}. For many known SPDEs on Banach spaces including stochastic 2D Navier-Stokes equations, the associated transition semigroups are all equicontinuous indeed, which can be derived from the gradient estimate
   \beqn
      |\nabla P_t\varphi(x)| \leqslant C(x) \left(||\varphi||_{\infty} + ||\nabla\varphi||_{\infty} \right), \ \ \forall t\geqslant t_0. \label{eqGradient}
   \eeqn
If $C(x)$ is uniformly finite in a neighborhood of $x$, $|P_t\varphi(x) - P_t\varphi(y)|$ can be controlled by $d(x,y)$ uniformly for all $y$ near $x$ and all time $t\geqslant t_0$, i.e. the equicontinuity.

However, there exist non-equicontinuous semigroups, or it is too complicated to prove equicontinuity. For example, Funaki and Spohn \cite{Funaki} introduced the Ginzburg-Landau $\nabla \varphi$ interface model, which admits at most one shift-invariant measure. Their argument is to control the long time average of semigroup evolution on initial distributions in the $L^p$-Wasserstein distance (see \cite[Proposition 2.1]{Funaki}, where $p$ was chosen to be $2$), formally like
   \beqn
      \limsup\limits_{T\to \infty} \frac{1}{T} \int_0^T W_p(P_t^*\mu, P_t^*\nu)^p dt \leqslant C W_p(\mu, \nu)^p. \label{eqFunaki}
   \eeqn
To authors' knowledge, it seems hopeless to improve (\ref{eqFunaki}) uniformly in $T$, because the individual ergodic theorem played a fundamental role in their proof so that it could not give any uniform rate of convergence for arbitrary $\mu$ and $\nu$. Some new developments on unique ergodicity of Ginzburg-Landau $\nabla \varphi$ interface model can be found in Cotar, Deuschel and M\"{u}ller \cite{CDM}, Cotar and Deuschel \cite{CD}.

Therefore, to reach the unique ergodicity for various models, we can attempt to skirt around the equicontinuity and asymptotic strong Feller property. What's more in the theoretical sense, these criteria stay far from necessity to prove the existence of invariant distributions. So in this paper, our purpose is to find sharp criteria for the ergodicity and asymptotic stability of {\em Feller semigroups} on Polish metric spaces with full generality. To this end, we will introduce some new notions, especially the {\em eventual continuity of semigroups} (see (\ref{eqPEvCo}) and (\ref{eqQEvCo}) below), which is almost necessary to the ergodic behavior. When $X$ is a Banach space, the eventual continuity can be derived from a weak gradient estimate, formally like
   \beqn
      \ \ \ \ \limsup\limits_{t\to \infty} |\langle\nabla P_t\varphi(x), h\rangle| \ \leqslant \ C(x) \left(||\varphi||_{\infty} + ||\nabla\varphi||_{\infty} \right), \ \ \forall ||h||=1. \label{eqGradientWeak}
   \eeqn
Comparing (\ref{eqGradientWeak}) with (\ref{eqGradient}), one can see the new ingredient is there doesn't request a uniform $t_0$ for all $h$. In other words, the evolution allows sensitive dependence on directions. More generally, (\ref{eqFunaki}) provides a useful and essential approach to yield the eventual continuity of semigroups on Polish metric spaces.

\bigskip

The basic setting is as follows. Let $X$ be a Polish space equipped with a Feller transition kernel $P(x,dy)$ on the Borel $\sigma$-field $\mathcal{B}$ satisfying $P(x,X) \equiv 1$, and $\mathrm{C_b}$ the set of bounded continuous functions on $X$. More precisely, the Feller property means $Pf\in \mathrm{C_b}$ for any $f\in \mathrm{C_b}$. In many cases, it is enough to replace $\mathrm{C_b}$ by a separable subalgebra $\mathcal{F}$, namely a subalgebra in $\mathrm{C_b}$ separating all the points in $X$. For instance, all bounded Lipschitz functions form a separable subalgebra.

We call a probability measure $\nu$ an {\em invariant distribution} if $\nu P = \nu$, where define $\nu P(f) = \nu(Pf) := \int Pf d\nu$. For simplicity, denote by $O_z$ a neighborhood of $z$, $Q^m = \frac{1}{m} \sum_{n=1}^m P^n$ the $m$-th average kernel, and $\overset{\textrm{w}}{\longrightarrow}$ the weak convergence.

Let's give an overview of main topics and results in this paper.

\bigskip

\noindent {\bf Equicontinuity vs. eventual continiuty}

The equicontinuity of $P^n$ at some $z\in X$ means for any $f\in \mathrm{C_b}$ (or a separable subalgebra $\mathcal{F}$)
   \[ \inf\limits_{O_z} \ \sup\limits_{n\geqslant 1}\ \sup\limits_{y\in O_z} \ |P^nf(y) - P^nf(z)| = 0. \]
Note that for Feller semigroups, this is equivalent to
   \beqn
     \inf\limits_{O_z} \ \limsup\limits_{n\to \infty}\ \sup\limits_{y\in O_z} \ |P^nf(y) - P^nf(z)| = 0, \label{eqSz05}
   \eeqn
which implies that $Q^m$ is equicontinuous at $z$ too
   \beqn
     \inf\limits_{O_z} \ \limsup\limits_{m\to \infty}\ \sup\limits_{y\in O_z} \ |Q^mf(y) - Q^mf(z)| = 0. \label{eqSz06}
   \eeqn
Roughly speaking, the equicontinuity describes that the orbits starting from a small ball should keep close to each other uniformly in time.

We define that $P^n$ is {\em eventually continuous at $z$} if for any $f\in \mathrm{C_b}$ (or $\mathcal{F}$)
   \beqn
     \inf\limits_{O_z}\ \sup\limits_{y\in O_z}
         \limsup\limits_{n\to\infty} \ |P^nf(y) - P^nf(z)| = 0. \label{eqPEvCo}
   \eeqn
which implies that $Q^m$ is eventually continuous at $z$ too
   \beqn
     \inf\limits_{O_z}\ \sup\limits_{y\in O_z}  \limsup\limits_{m\to\infty} \ |Q^mf(y) - Q^mf(z)| = 0. \label{eqQEvCo}
   \eeqn
Thus $\limsup\limits_{n\to \infty} P^nf$ and $\limsup\limits_{m\to \infty} Q^mf$ are continuous at $z$, and so are their inferior limits respectively. That is why we call it the eventual continuity. In comparison with the equicontinuity, clearly hold
   \[ (\ref{eqSz05}) \ \ \Rightarrow\ \ (\ref{eqPEvCo}),\ \ \ \ \ (\ref{eqSz06}) \ \ \Rightarrow\ \ (\ref{eqQEvCo}),\]
but not vice versa. The eventual continuity is much weaker since there is no uniform restriction on time.
\begin{remark}
Here is a toy example to show that an eventual continuous semigroup can be not equicontinuous. Let $\mathbb{H}$ be a Hilbert space with a basis $\{h_n\}_{n\geqslant 1}$. Let
  \[ \lambda_{n,k} = \left\{\begin{array}{ll}
            2^{k-n},  & 0\leqslant k < n; \\
            k - n + 1, & n \leqslant k < 2n; \\
            (k - 2n +\pi)^{-1},   & 2n \leqslant k < 2^n.
        \end{array} \right. \]
Denote by $x_{n,k} = \lambda_{n,k} h_n$, $X_n = \{ x_{n,k}: 0\leqslant k \leqslant 2^n \}$ the $n$-th state set in direction $h_n$, and $X = \{0\} \cup X_1 \cup X_2 \cup \cdots$ the state space. Define the transition
  \[ P(0,0) = 1, \ \ \ P(x_{n,\; k\;\mathrm{mod}\;2^n},\; x_{n,\; k+1\;\mathrm{mod}\;2^n}) = 1. \]
Each $X_n$ is a cycle. Then $Q_m$ is eventually continuous at $0$, but not equicontinuous.
\end{remark}

\begin{remark}
Most recently, Prof. N. Bouleau told us that the notion of eventual continuity is adapted to a general framework of sticky convergence or sticking topology, which was presented in his early work \cite{Bouleau} (or see \cite{Bouleau2}). The sticky convergence is finer than the pointwise convergence and coarser than the locally uniform convergence, which gives the coarsest topology preserving continuity.
\end{remark}

Let's point out, in view of the mean ergodic theorem, (\ref{eqQEvCo}) is almost necessary to the existence of ergodic measures except negligible sets. Hence, it is reasonable to call (\ref{eqQEvCo}) a sharp condition in the sense that there would be no information on negligible set of invariant distribution prior to we could prove its existence.

\bigskip

\noindent {\bf Existence of invariant distributions}

First of all, let's recall a fundamental characterization of the existence of invariant distributions, for example, see \cite[Proposition 3.1]{Lasota-Szarek}.

\begin{proposition} {\em (\cite{Lasota-Szarek})} \label{propCha}
Suppose there exist some point $x$ and compact set $K$ such that
  \[ \limsup\limits_{m\to \infty} Q^m(x,K) > 0. \]
Then there exists an invariant distribution.
\end{proposition}

There are some other ``compact"-type criteria. For example, Prof. M.-F. Chen \cite[Theorem 4.11]{MFChen} offered one for the existence of stationary distributions for Feller semigroups. Say a nonnegative measurable real-valued function $h$ is compact if the set $\{x: h(x) \leqslant c\}$ is compact for all $c\geqslant 0$. Then there exists a stationary distribution provided there exist a compact $h$, some point $x_0$ and a constant $C\geqslant 0$ such that $\frac{1}{m}\sum\nolimits_{n=1}^m P^nh(x_0) \leqslant C$ for all $m\geqslant 1$.

However, to study complicated models on Polish spaces, it is too difficult to determine a compact set under the infinite dimensional topology. In general, it is much more natural and useful to replace $K$ by a neighborhood $O$. For equicontinuous semigroups, Szarek \cite[Proposition 2.1]{SzarekFeller} made some notable improvements. According to \cite{SzarekFeller}, say a {\em lower bound condition} holds at $z$, if for any neighborhood $O_z$, there exists some $x$ such that
 \begin{align}
   \limsup\limits_{m\to \infty} Q^m(x,O_z) > 0. \tag{$\mathcal{L}$}
 \end{align}

\begin{proposition} {\em (\cite{SzarekFeller})} \label{propCha2}
Suppose $P^n$ is equicontinuous at some $z$ and satisfies $(\mathcal{L})$ at $z$ too.
Then the sequence $\{Q^m(z, \cdot)\}_{m\geqslant 1}$ is tight.
\end{proposition}
\begin{remark}
The tightness implies there is a compact $K$ with $\limsup\limits_{m\to \infty} Q^m(z,K) > 0$, which yields the existence of invariant distributions by Proposition \ref{propCha}.
\end{remark}

Condition $(\mathcal{L})$ is necessary to the existence of invariant distributions, but the equicontinuity (see (\ref{eqSz05}) or (\ref{eqSz06})) is not. In this paper, we deal with general Feller semigroups with eventual continuity rather than equicontinuous ones. We give two sharp criteria.

\begin{theorem} \label{thmFeller2}
Suppose $Q^m$ is eventually continuous at some $z$, and satisfies for any $O_z$
 \begin{align}
   \limsup\limits_{m\to \infty} Q^m(z,O_z) > 0. \tag{$\mathcal{L}_S$}
 \end{align}
Then the sequence $\{Q^m(z, \cdot)\}_{m\geqslant 1}$ is tight.
\end{theorem}
\begin{remark}
$(\mathcal{L}_S)$ is a necessary condition for the existence of invariant distributions too, even if it is a bit stronger than $(\mathcal{L})$.
\end{remark}

The next criterion has no restrictions on starting points of transitions.
\begin{theorem} \label{thmFeller}
Suppose $Q^m$ is eventually continuous at some $z$ and satisfies the following property:
\begin{align}
  \textrm{the lower bound condition holds at every point in a neighborhood of } z.  \tag{$\mathcal{L}_{loc}$}
\end{align}
Then the sequence $\{Q^m(z, \cdot)\}_{m\geqslant 1}$ is tight.
\end{theorem}
\begin{remark} \label{remStrongLower}
(i) If there exists an invariant distribution $\mu$, we define $Y=\mathrm{Supp}\mu$ with the relative topology. Then $(\mathcal{L}_{loc})$ is true for the whole $Y$. In other words, $(\mathcal{L}_{loc})$ is a necessary condition on the ergodic component. (ii) By the Feller property, it is sufficient to assume $(\mathcal{L})$ at $z$, and assume for any $O_y \subset O_z$, there exists a time $n$ (depending on $y$) such that $P^n(z,O_y)>0$. Then $(\mathcal{L}_{loc})$ holds at $z$.
\end{remark}

\bigskip

\noindent {\bf Uniqueness of invariant distribution}

It's known that there exists an ergodic measure if one can find an invariant distribution, see Hille and Worm \cite[Corollary 4.8]{HilleWorm} . For this reason,  we discuss the uniqueness of ergodic measure now.

According to \cite{Hairer}, an increasing sequence of (pseudo) metrics $d_i$ on $X$ is called a {\em totally separating system} if $\lim\limits_{n\to \infty} d_i(x,y) = 1$ for all $x\neq y$. Denote
   \[ \parallel\varphi\parallel_{d_i} = \sup\limits_{x\neq y} \frac{|\varphi(x) - \varphi(y)|}{d_i(x,y)}, \ \ \
       \parallel\mu - \nu\parallel_{d_i} = \sup\limits_{\parallel\varphi\parallel_{d_i}\leqslant 1} \left| \int \varphi d\mu - \int \varphi d\nu \right|. \]
Say $P^n$ is {\em asymptotic strong Feller} at $z$ if there exists a totally separating system $\{d_i\}$ and a sequence $n_i >0$ such that
   \beqn
      \inf\limits_{r> 0} \ \limsup\limits_{n\to\infty} \ \sup\limits_{y\in B(z,r)}
         \ \parallel P^{n_i}(y, \cdot) - P^{n_i}(z, \cdot) \parallel_{d_i} = 0. \label{eq701}
   \eeqn
It can be applied to show that, if $\mu$ and $\nu$ are two distinct ergodic measures, then $z \notin \mathrm{Supp}\mu \cap \mathrm{Supp}\nu$.

Clearly, (\ref{eq701}) implies $P^{n_i}$ is eventually continuous at $z$ with respect to the $d_i$-Lipschitz test functions. This fact might have nothing to do with the ergodicity if these $n_i$ are selected irregularly. However, consider the average kernel, we have

\begin{proposition} \label{propSepa}
Suppose a subsequence $Q^{m_i}$ is eventually continuous at $z$. Then $z \notin \mathrm{Supp}\mu \cap \mathrm{Supp}\nu$, when $\mu$ and $\nu$ are two distinct ergodic measures.
\end{proposition}

Consequently, the uniqueness of ergodic measures can be derived from eventual continuity on the whole $X$ combining with the {\em weak type of irreducibility}, i.e. for any $x_1$ and $x_2$, there exists some $y$ such that for any neighborhood $O_y$, there exist $n_1$ and $n_2$ respectively with $P^{n_1}(x_1,O_y)>0$ and $P^{n_2}(x_2,O_y)>0$.

In summary, the eventual continuity (\ref{eqQEvCo}) can make contributions to both the existence and uniqueness of invariant distribution.

\bigskip

\noindent {\bf Asymptotic stability}

Now, suppose $(X,\mathcal{B},P)$ admits an ergodic measure $\mu$. Simply write $X_\mu = \mathrm{Supp}\mu$.

For equicontinuous semigroups, Szarek \cite[Theorem 2]{Szarektracer} proved that

\begin{proposition} {\em (\cite{Szarektracer})} \label{propCha3}
Suppose $P^n$ is equicontinuous on $X$, and there exists some $z\in X_\mu$ such that for any neighborhood $O_z$
 \beqn
   \liminf\limits_{n\to \infty} P^n(z,O_z) > 0.  \label{eqliminflower}
 \eeqn
Then $P^n(x,\cdot) \overset{\mathrm{w}}{\longrightarrow} \mu$ for all $x\in X_\mu$.
\end{proposition}

This result can be essentially improved for general Feller semigroups. First of all, let's introduce a notion of {\em aperiodicity}, which is much weaker than (\ref{eqliminflower}). Say $z$ is aperiodic, if for any $O_z$, there exists $N$ such that $P^n(z,O_z) > 0$ for all $n\geqslant N$. We prove that

\begin{theorem} \label{thmAsyStaSup}
Suppose $P^n$ is eventually continuous on $X_\mu$. The next two statements are equivalent:
\begin{enumerate}
 \item[$(1)$] $P^n(x,\cdot) \overset{\mathrm{w}}{\longrightarrow} \mu$ for all $x\in X_\mu$.
 \item[$(2)$] $X_\mu$ contains an aperiodic point.
\end{enumerate}
\end{theorem}
\begin{remark}
The asymptotic stability on $X_\mu$ implies a restricted eventual continuity on the subspace $(X_\mu, \mathcal{B}_{|X_\mu}, P_{|X_\mu})$ with relative topology. Hence, the eventual continuity is necessary to the asymptotic stability on $(X_\mu, \mathcal{B}_{|X_\mu}, P_{|X_\mu})$.
\end{remark}

It's known that $X_\mu = X$ can be derived from the {\em topological irreducibility}, i.e. for any $x,y\in X$ and any $O_y$, there exists some $n$ with $P^n(x,O_y)>0$. However, sometimes the topological irreducibility might be not true. In general, we provide a criterion for the global asymptotic stability.

\begin{theorem} \label{thmAsyStaGloNew}
Suppose that $P^n$ is eventually continuous on $X_\mu$ and $P^n(x,\cdot) \overset{\mathrm{w}}{\longrightarrow} \mu$ for all $x\in X_\mu$. The next two statements are equivalent:
\begin{enumerate}
\item[$(1)$] $P^n(x,\cdot) \overset{\mathrm{w}}{\longrightarrow} \mu$ for all $x\in X$.
\item[$(2)$] there exists $z\in X_\mu$ such that for any $O_z$, there exists $\eta > 0$ satisfying $\inf\limits_{x\in X}\limsup\limits_{n\to \infty} P^n(x, O_z) \geqslant \eta.$
\end{enumerate}
\end{theorem}

We give another criterion for the case that one can prove a local topological irreducibility around some $z\in X_\mu$ (thus $z$ becomes an inner point in $X_\mu$).
\begin{theorem} \label{thmAsyStaGlo}
Suppose that $P^n$ is eventually continuous on $X_\mu$ and $P^n(x,\cdot) \overset{\mathrm{w}}{\longrightarrow} \mu$ for all $x\in X_\mu$. Suppose also
\begin{enumerate}
\item[$(\mathbf{A}1)$] there exists an inner point $z\in X_\mu$;
\item[$(\mathbf{A}2)$] for any bounded set $A$ and any $O_z$, there exists $\eta>0$ such that for any $x\in A$, there exists $k$ such that $P^k (x,O_z) \geqslant \eta$;
\item[$(\mathbf{A}3)$] for any $x\in X$ and $\varepsilon>0$, there exists a bounded set $B$ and a subsequence $n_i\to \infty$ such that $\limsup\limits_{i\to\infty} P^{n_i}(x,B) \geqslant 1-\varepsilon$.
\end{enumerate}
Then $P^n(x,\cdot) \overset{\mathrm{w}}{\longrightarrow} \mu$ for all $x\in X$.
\end{theorem}
\begin{remark}
$(\mathbf{A}2)$ and $(\mathbf{A}3)$ are necessary. The new ingredient is one doesn't need to find a uniform lower bound for $\limsup\limits_{n\to \infty} P^n(x,O_z)$ as (2) in Theorem \ref{thmAsyStaGloNew}.
\end{remark}

\bigskip

This paper is arranged as follows. In Section \ref{invdis}, we discuss the existence and uniqueness of invariant distribution. Sections \ref{ASSupp} and \ref{ASGlo} are respectively devoted to the asymptotic stability on the ergodic support and whole state space. In Section \ref{2DNS}, we would like to simply revisit the unique ergodicity and prove the asymptotic stability of stochastic 2D Navier-Stokes equations according to our criteria, based on partial estimates from \cite{Mattingly} and \cite{Hairer}.

Let's point out, all the notions and results in Sections \ref{invdis}, \ref{ASSupp} and \ref{ASGlo} can be freely extended to continuous-time semigroup $P_t$ correspondingly.

\bigskip

\section{Invariant distributions}
 \label{invdis}

In this section, endow $(X,\mathcal{B})$ with a metric $\rho$ since $X$ is metrizable. Denote by $B(x,r)=\{y: \rho(x,y)< r\}$ the open ball of radius $r$ centered at $x$, and $A^\varepsilon = \{x: \rho(x,A)< \varepsilon\}$ the $\varepsilon$-neighborhood of a set $A$.

\subsection{Some lemmas}

\begin{lemma} \label{lemDisj}
Let $\{A_n^\varepsilon\}_{n\geqslant 1}$ be a sequence of mutually disjoint $\varepsilon$-neighborhoods of $A_n$. Then, for any compact set $C$, there exists $N>0$ such that for all $n\geqslant N$
    \[ C \cap A_n^{\varepsilon/2} = \emptyset .  \]
\end{lemma}
\begin{proof}
Assume there exist $x_n \in C \cap A_n^{\varepsilon/2}$ for infinitely many $n$. Without loss of
generality, assume that $x_n$ tends to some $x\in C$. Then, we
have $x\in A_n^\varepsilon$ for any $n$ with $\rho(x_n,x) < \varepsilon/2$, which contradicts the mutual disjointness.
\end{proof}

\begin{lemma} \label{lemDeltaGamma}
Let $\{A_n^\varepsilon\}_{n\geqslant 1}$ be a sequence of mutually disjoint $\varepsilon$-neighborhoods of compact $A_n$. Let $x\in X$ and $m_i \in \mathbb{N}$. Then for any $\eta>0$, there exists $N >0$ such that for all $n\geqslant N$
  \[  \liminf\limits_{i\to \infty} Q^{m_i}(x, A_n^{\varepsilon/4}) \leqslant \eta. \]
\end{lemma}
\begin{proof}
Define $\delta = \sup\limits_{C \textrm{ compact}} \liminf\limits_{i\to \infty} Q^{m_i}(x, C^{\varepsilon/4})$.
Choose some $\gamma\geqslant 0$ and compact subset $C$ to satisfy
  \[ 0\leqslant \delta - \gamma \leqslant \eta, \ \ \ \
      \liminf\limits_{i\to \infty} Q^{m_i}(x, C^{\varepsilon/4}) \geqslant \gamma. \]
By Lemma \ref{lemDisj}, there exists some $N$ such that $C^{\varepsilon/4} \cap A_n^{\varepsilon/4} =
\emptyset$ for all $n\geqslant N$. Then, we have by the definition of $\delta$ that
  \[ \delta \geqslant \liminf\limits_{i\to \infty} Q^{m_i}(x, C^{\varepsilon/4} \cup A_n^{\varepsilon/4})
            \geqslant \gamma +  \liminf\limits_{i\to \infty} Q^{m_i}(x, A_n^{\varepsilon/4}),\]
which implies $\eta \geqslant \liminf\limits_{i\to \infty} Q^{m_i}(x, A_n^{\varepsilon/4})$ for all $n\geqslant N$.
\end{proof}

\begin{lemma} \label{lemLiminf}
Let $A_n \in \mathcal{B} \; (n\geqslant 1)$ be a sequence of mutually disjoint sets. Let $x\in X$ and $m_i\geqslant 1$. Then for any $\varepsilon > 0$, there is $N>0$ such that $\liminf\limits_{i\to
\infty} Q^{m_i}(x, A_N) \leqslant \varepsilon$. Hence, there exists $m_{i_k}$ with
$\limsup\limits_{k\to \infty} Q^{m_{i_k}}(x, A_N) \leqslant \varepsilon$.
\end{lemma}
\begin{proof}
Due to $1 \geqslant \liminf\limits_{i\to \infty} Q^{m_i}(x, \bigcup\limits_{n} A_n) \geqslant \sum\limits_{n} \liminf\limits_{i\to \infty} Q^{m_i}(x, A_n)$.
\end{proof}

\subsection{Proof of Theorem \ref{thmFeller}}

The basic idea is partially from \cite{SzarekFeller}, but since we deal with the eventual continuity rather than equicontinuity, the proof becomes much more difficult here. For ease of reading, one can check firstly Step 1 and Step 4 provided (\ref{eq415}) below.

\begin{proof}
Referring to Billingsley \cite{Billingsley}, it is enough to show for any $\varepsilon >0$, there exists a compact set $E$ such that $Q^m(z,E^{2\varepsilon}) \geqslant 1-2\varepsilon$ for all $m\geqslant 1$. We divide the proof into four steps.

{\bf Step 1}. Given $\varepsilon > 0$. Denote $K_0 = \{z\}$ and $n_0 = 1$. By induction, we will find $K_j$ and $n_j$ as follows. For each $j\geqslant 1$, there exists some compact set $E_j$ satisfying
   \[ E_j\supset \bigcup\limits_{0\leqslant l<j} K_l, \ \ \textrm{ and } \ \
       \min\limits_{m \leqslant n_{j-1}} Q^m(z, E_j^{2\varepsilon}) \geqslant 1-2\varepsilon. \]
Let's introduce
  \beqn
     \inf\limits_{m > n_{j-1}} Q^m(z, E_j^{2\varepsilon}) =: 1 - 2\theta_j . \label{eqb10}
  \eeqn
If $\theta_j = 0$, we stop the procedure; else there exists $n_j > n_{j-1}$ such that $Q^{n_j}(z,E_j^{2\varepsilon}) < 1 - \theta_j$, together with a compact set $K_j$ satisfying
  \beqn
     E_j^{2\varepsilon} \cap K_j = \emptyset, \ \ \textrm{ and } \ \
       Q^{n_j}(z, K_j) \geqslant \theta_j. \label{eqb11}
  \eeqn

Either we can finish the proof just in finite steps, or collect a sequence of data $\{E_j, \theta_j, n_j, K_j\}$. Clearly, all $K_j^\varepsilon$ are disjoint mutually. Define
  \[ f_j(y) = \rho(y, (K_j^{\varepsilon/4})^c) \Big{/} \left( \rho(y,(K_j^{\varepsilon/4})^c) + \rho(y, K_j) \right),\]
which fulfills that $||f_j||_{\mathrm{Lip}} \simeq 4/\varepsilon$ and $\mathbf{1}_{K_j} \leqslant f_j \leqslant \mathbf{1}_{K_j^{\varepsilon/4}}$. \\

{\bf Step 2}. We want to select a subsequence from the above data such that it is so sparse that (\ref{eq413}) below holds. This step will be cut into four parts.

{\bf Part 2.1}. Denote $z_0 = z$, $s_0=r$ and $B_0 = B(z_0,s_0)$. By Condition $(\mathcal{L}_{loc})$, there exist some $x_0$ and $\{m_{i,0}\}_{i\geqslant 1}$ such that $\alpha_0 := \lim\limits_{i\to \infty}
Q^{m_{i,0}}(x_0,B_0) > 0$.

Using Lemma \ref{lemDeltaGamma} yields a big $j_0$ such that
    \beqn
       \varepsilon\alpha_0/16 &\geqslant&
        \liminf\limits_{i\to \infty} Q^{m_{i,0}}(x_0, K_{j_0}^{\varepsilon/4})
        \ =\ \liminf\limits_{i\to \infty} \int Q^{n_{j_0}}(y, K_{j_0}^{\varepsilon/4}) Q^{m_{i,0}}(x_0,dy) \nonumber\\
        &\geqslant& \liminf\limits_{i\to \infty} \int_{B_0} Q^{n_{j_0}}f_{j_0}(y) Q^{m_{i,0}}(x_0,dy). \label{eq390}
    \eeqn
Due to the Feller property, define an open subset in $B_0$ as
  \[ A_0 = \{y\in B_0: Q^{n_{j_0}}f_{j_0}(y) < \varepsilon/8\}. \]
It follows from (\ref{eq390}) that
  \[ \varepsilon \alpha_0/16 \geqslant \liminf\limits_{i\to \infty} \int_{B_0 - A_0} Q^{n_{j_0}}f_{j_0}(y) Q^{m_{i,0}}(x_0,dy), \]
which implies $\liminf\limits_{i\to \infty} Q^{m_{i,0}}(x_0, B_0 - A_0) \leqslant \alpha_0/2$, and thus
  \beqn
   \limsup\limits_{i\to \infty} Q^{m_{i,0}}(x_0,A_0) \geqslant \alpha_0/2. \label{eq392}
  \eeqn
So $A_0$ is nonempty, which contains a ball $B_1$ of radius less than $r/2$ such that
  \[  Q^{n_{j_0}}f_{j_0}(y) \leqslant \varepsilon/8, \ \ \forall y\in \overline{B_1}. \]

Inductively for each $B_k$ ($k\geqslant 0$), Condition $(\mathcal{L}_{loc})$ yields some $x_k$ and $\{m_{i,k}\}_{i\geqslant 1}$ with
  \[ \alpha_k := \lim\limits_{i \to \infty}Q^{m_{i,k}}(x_k, B_k) > 0.\]
For the same reason, there exist $j_k$ and a ball $B_{k+1} \subset B_k$ of radius less than $r/{2^{k+1}}$ satisfying
  \[  Q^{n_{j_k}}f_{j_k}(y) \leqslant \varepsilon/8, \ \ \forall y\in \overline{B_{k+1}}. \]
Hence, we can find a common $y_0 \in \bigcap \overline{B_k}$ such that
   \beqn
       Q^{n_{j_k}}f_{j_k}(y_0) \leqslant \varepsilon/8, \ \ \forall k\geqslant 0. \label{eq404}
   \eeqn

{\bf Part 2.2}. By Lemma \ref{lemLiminf}, there is a
subsequence $\{j_{k,1}\}_{k\geqslant 0} \subset \{j_k\}$ such that
  \[ \limsup\limits_{k\to \infty} Q^{n_{j_{k,1}}} f_{j_0}(y_0) \leqslant \varepsilon/8. \]
For the same reason, there is $\{j_{k,2}\}_{k\geqslant 0} \subset \{j_{k,1}\}$ such that
  \[ \limsup\limits_{k\to \infty} Q^{n_{j_{k,2}}} f_{j_{0,1}}(y_0) \leqslant \varepsilon/(8\cdot 2). \]
By induction, we have $\{j_{k,l+1}\}_{k\geqslant 0} \subset \{j_{k,l}\}$ satisfying
  \[ \limsup\limits_{k\to \infty} Q^{n_{j_{k,l+1}}} f_{j_{0,l}}(y_0) \leqslant \varepsilon/(8\cdot 2^l). \]

For simplicity of natation, still use $j_l$ instead of $j_{0,l}$. Recall (\ref{eq404}), we obtain
  \beqn
     Q^{n_{j_l}} f_{j_l}(y_0) \leqslant \varepsilon/8, \ \ \ \limsup\limits_{k\to \infty} Q^{n_{j_k}} f_{j_l}(y_0) \leqslant \varepsilon/(8\cdot 2^l), \ \ \forall l\geqslant 0.
    \label{eq408}
  \eeqn

{\bf Part 2.3}. Since $K_j^\varepsilon$ are disjoint mutually, there exists a big $u$ such that
   \[ Q^{n_{j_0}} \sum\limits_{k\geqslant u} f_{j_k}(y_0) \leqslant \varepsilon/8. \]
Combining with the first inequality in (\ref{eq408}) yields
   \[ Q^{n_{j_0}} (f_{j_0} + \sum\limits_{k\geqslant u} f_{j_k}) (y_0) \leqslant \varepsilon/4. \]

Let $\hat{j}_0 = j_0$ and $\hat{j}_1 = j_u$. For the same reason, there exists $v>u$ such that
   \[ Q^{n_{\hat{j}_1}} (f_{\hat{j}_1} + \sum\limits_{k\geqslant v} f_{j_k}) (y_0) \leqslant \varepsilon/4. \]
Let $\hat{j}_2 = j_v$. By induction, we have $\{\hat{j}_l\}_{l \geqslant 0} \subset \{j_l\}$ such that for all $l\geqslant 0$
  \[  Q^{n_{\hat{j}_l}} \sum\limits_{k\geqslant l} f_{\hat{j}_k}(y_0) \leqslant \varepsilon/4. \]

For simplicity of notation, still use $j_l$ instead of $\hat{j}_l$. Besides the second inequality in (\ref{eq408}), we obtain
  \beqn
    Q^{n_{j_l}} \sum\limits_{k\geqslant l} f_{j_k}(y_0) \leqslant \varepsilon/4, \ \ \limsup\limits_{k\to \infty} Q^{n_{j_k}} f_{j_l}(y_0) \leqslant \varepsilon/(8\cdot 2^l), \ \ \forall l\geqslant 0. \label{eq4111}
  \eeqn

{\bf Part 2.4}. Based on the second inequality in (\ref{eq4111}), there exists a big $u$ such that
   \[ Q^{n_{j_u}} f_{j_0}(y_0) \leqslant \varepsilon/4, \ \ \
       \limsup\limits_{k\to \infty} Q^{n_{j_k}} (f_{j_0} + f_{j_u}) (y_0) \leqslant (1+2^{-1}) \cdot \varepsilon/8. \]
For the same reason, there exists a big $v > u$ such that
   \[ Q^{n_{j_v}} (f_{j_0} + f_{j_u}) (y_0) \leqslant \varepsilon/4,\ \
      \limsup\limits_{k\to \infty} Q^{n_{j_k}} (f_{j_0} + f_{j_u} + f_{j_v}) (y_0) \leqslant (1+2^{-1}+2^{-2}) \cdot \varepsilon/8. \]

Let $\check{j}_0 = j_0, \check{j}_1 = j_u, \check{j}_2 = j_v$. By induction, we have $\{\check{j}_l\}_{l\geqslant 0} \subset \{j_l\}$ satisfying
   \[  Q^{n_{\check{j}_l}} \sum\limits_{0\leqslant k < l} f_{\check{j}_k}(y_0) \leqslant \varepsilon/4, \ \ \forall l\geqslant 0. \]
Combining with the first inequality in (\ref{eq4111}) yields
  \[ Q^{n_{\check{j}_l}} \sum\limits_{k\geqslant 0} f_{\check{j}_k}(y_0) \leqslant \varepsilon/2, \ \ \forall l\geqslant 0. \]

For simplicity of notation, still use $j_l$ instead of $\check{j}_l$, namely
  \beqn
    Q^{n_{j_l}} \sum\limits_{k\geqslant 0} f_{j_k}(y_0) \leqslant \varepsilon/2, \ \ \forall l\geqslant 0. \label{eq413}
  \eeqn

{\bf Step 3}. Write $j_{k,0} = j_k$. Let's repeat Step 2 by substituting $s_0$ to $s_1 = r/2$, then obtain some $\{j_{k,1}\}_{k\geqslant 0} \subset \{j_{k,0}\}$ and $y_1 \in \overline{B(z,s_1)}$ such that (similar to (\ref{eq413}))
  \[ Q^{n_{j_{l,1}}} \sum\limits_{k\geqslant 0} f_{j_{k,1}} (y_1) \leqslant \varepsilon/2. \]
By induction, we obtain the $p$-th subsequence $\{j_{k,p}\}_{k\geqslant 0} \subset \{j_{k,p-1}\}$ and some $y_p \in \overline{B(z, s_p)}$ for $s_p = r/2^p$ such that
  \[ Q^{n_{j_{l,p}}} \sum\limits_{k\geqslant 0} f_{j_{k,p}} (y_p) \leqslant \varepsilon/2, \ \ \forall l\geqslant 0. \]

Denote $\tilde{j}_p = j_{0,p}$, it follows that
  \beqn
    Q^{n_{\tilde{j}_l}} \sum\limits_{k\geqslant p} f_{\tilde{j}_k}(y_p) \leqslant \varepsilon/2, \ \ \forall l\geqslant p. \label{eq415}
  \eeqn
In fact, (\ref{eq415}) plays a crucial role for the proof. Note that $y_p \to z$.\\

{\bf Step 4}. Let $j^*_0 = \tilde{j}_0$. The eventual continuity yields some $r^*_0$ such that for all $y\in B(z,r^*_0)$
   \[ \limsup\limits_{m\to \infty} |Q^m f_{j^*_0} (z) - Q^m f_{j^*_0} (y)| \leqslant \varepsilon/8. \]
Due to (\ref{eq415}), choose $y_p\in B(z,r^*_0)$, denoted by $y^*_0$. And for this $p$, choose some $j^*_1 \in \{\tilde{j}_k\}$ with $j^*_1\geqslant \tilde{j}_p$.

By induction, if we have $j^*_0,j^*_1, \ldots, j^*_u$, there is $r^*_u$ such that for all $y\in B(z,r^*_u)$
   \beqn \limsup\limits_{m\to \infty} |Q^m \sum\limits_{l\leqslant u} f_{j^*_l} (z)
       - Q^m \sum\limits_{l\leqslant u} f_{j^*_l} (y)| \leqslant \varepsilon/8. \label{eq382}
   \eeqn
Choose $y_q\in B(z,r^*_u)$, denoted by $y^*_u$, and then $j^*_{u + 1} \in \{\tilde{j}_k\}$ with $j^*_{u + 1}\geqslant \tilde{j}_q$.

Consider the subsequence $\{j^*_k\}_{k\geqslant 0} \subset \{\tilde{j}_k\}$. Define $g = \sum f_{j^*_k} \in \mathrm{Lip_b}$. Again, the eventual continuity yields some $r^*$ such that for all $y\in B(z,r^*)$
   \beqn
     \limsup\limits_{m\to \infty} |Q^m g(z) - Q^m g(y)| \leqslant \varepsilon/8. \label{eq384}
   \eeqn
Fix some $y^*_u\in B(z,r^*)$ (for big $u$), denote $g_0 = \sum\limits_{l \leqslant u} f_{j^*_l}$ and $g_1 = g - g_0$. We have
   \beq
     && \limsup\limits_{k\to \infty} |Q^{n_{j^*_k}} g_1(z) - Q^{n_{j^*_k}} g_1 (y^*_u)| \\
     &\leqslant&  \limsup\limits_{k\to \infty} |Q^{n_{j^*_k}} g(z) - Q^{n_{j^*_k}} g (y^*_u)|
       + \limsup\limits_{k\to \infty} |Q^{n_{j^*_k}} g_0 (z) - Q^{n_{j^*_k}} g_0(y^*_u)|,
   \eeq
which is less than $\varepsilon/4$ by (\ref{eq382}-\ref{eq384}). Combining (\ref{eq415}), we have
  \[ \limsup\limits_{k\to \infty} Q^{n_{j^*_k}} g_1(z) \leqslant
     \limsup\limits_{k\to \infty} Q^{n_{j^*_k}} g_1 (y^*_u) + \varepsilon/4 \leqslant 3\varepsilon/4. \]

Recall (\ref{eqb10}-\ref{eqb11}) in Step 1, since for any $k>u$
  \[ \theta_{j^*_k} \leqslant Q^{n_{j^*_k}} (z, K_{j^*_k}) \leqslant Q^{n_{j^*_k}} f_{j^*_k} (z) \leqslant
      Q^{n_{j^*_k}} g_1(z), \]
it follows $\limsup\limits_{k\to \infty} \theta_{j^*_k} \leqslant 3\varepsilon/4$.
So there exist $\theta_{j^*_k} \leqslant \varepsilon$ and a compact $E_{j^*_k}$ such that
  \[ Q^m(z, E_{j^*_k}^{2\varepsilon}) \geqslant 1 - 2\varepsilon, \ \ \forall m\geqslant 1.\]
The proof of tightness is completed.
\end{proof}

\subsection{Proof of Theorem \ref{thmFeller2}}

Let's prove Theorem \ref{thmFeller2}.

\begin{proof}
Recall Part 2.1 in the proof of Theorem \ref{thmFeller}, the starting point is fixed as $z$ now. Using (\ref{eq392}), we can find some $n$ with $P^n(z,A_0)>0$, which implies there exists a ball $B_1 \subset A_0$ of radius less than $r/2$ such that $P^n(z,B_1)>0$.

By the Feller property, there exists $s>0$ such that for any $\xi \in B(z,s)$
  \[ P^n(\xi, B_1) \geqslant \frac12 P^n(z, B_1) >0.  \]
From Condition $(\mathcal{L}_S)$ at $z$, we derive
  \[ \limsup\limits_{m\to \infty} Q^m(z, B_1) \geqslant \frac12 P^n(z, B_1) \cdot \limsup\limits_{m\to \infty} Q^m(z, B(z,s))  > 0, \]
which can work for Part 2.2. Then we follow the remaining steps.
\end{proof}

\subsection{Proof of Proposition \ref{propSepa}}

Let's prove Proposition \ref{propSepa}.

\begin{proof}
Assume $z\in \mathrm{Supp}\mu \cap \mathrm{Supp}\nu$. Choose $f\in \mathrm{Lip_b}$ with $b = |\mu (f) - \nu (f)| > 0$. Then there exists $r>0$ satisfying $\lim\limits_{i\to \infty} |Q^{m_i}f(z) - Q^{m_i}f(y)| \leqslant b/4$ for all $y\in B(z,r)$. By the mean ergodic theorem, there exist $y_1$ and $y_2\in B(z,r)$ such that
  \[ \mu (f) = \lim\limits_{i\to \infty} Q^{m_i}f(y_1), \ \ \ \nu (f) = \lim\limits_{i\to \infty} Q^{m_i}f(y_2), \]
which implies $|\mu (f) - \nu (f)| \leqslant b/2$. This is a contradiction.
\end{proof}

\bigskip

\section{Asymptotic stability on support}
 \label{ASSupp}

In this section, we give the proof of Theorem \ref{thmAsyStaSup}.

\begin{lemma} \label{lemEveConSub}
Suppose $P^n$ is eventually continuous on $X_\mu$. Then for any $f\in \mathrm{C_b}$, there exist a sequence of compact set $K_i$ with $\mu(K_i) \uparrow 1$, and a subsequence $P^{n_k}f$ uniformly converging to some $g$ on each $K_i$. Moreover, extend $g$ to be $\limsup\limits_{k\to \infty} P^{n_k}f$ on whole $X$, then $g$ is continuous on $X_\mu$ and $P^ng = \lim\limits_{k\to\infty} P^{n+n_k}f$ $\mu$-a.e. on $X_\mu$.
\end{lemma}
\begin{proof}
Let $F \subset X_\mu$ be a compact set. For any $\varepsilon > 0$ and $x\in F$, due to the eventual continuity, there exists some $O_x$ such that for all $y\in O_x$
   \[ \limsup\limits_{n\to\infty} \ |P^nf(y) - P^nf(x)| \leqslant \varepsilon/8. \]
Thus there exists a sequence of increasing subsets $E_{x,m}$ (containing $x$) tending to $O_x$ such that for all $n\geqslant m$ and $y,y'\in E_{x,m}$
   \[ |P^nf(y) - P^nf(y')| \leqslant \varepsilon/2.  \]
By the Feller property, $E_{x,m}$ can be chosen as a closed subset. Since $F$ is compact, we can find a finite open covering $\{O_{x_1}, \ldots, O_{x_p}\}$ of $F$, and then select a big $m$ such that $\mu(F - \bigcup_{j=1}^p E_{x_j,m}) \leqslant \varepsilon \mu(F)/2$.

For convenience of notation, denote
  \[ p_0=1, \ \ F_0=F,\ \ p_1=p, \ \ F_{j,1} = F \cap E_{x_j,m}, \ \ F_1 = \bigcup\nolimits_{j=1}^{p_1} F_{j,1},
      \ \ x_{j,1} = x_j. \]
Here $F_{j,1}$ is still compact. By induction, we have a net-like structure $\{F_{j,l}, p_l\}$ satisfying $1\leqslant j\leqslant p_l$ and
\begin{enumerate}
 \item $|P^nf(y) - P^nf(y')| \leqslant \varepsilon/2^l, \ \ \forall \textrm{ big }n ,\ \forall\; y,y'\in F_{j,l}$;
 \item $\forall j, \ \exists i, \textrm{ s.t. }
  F_{j,l} \subset F_{i,l-1}$;
 \item $\mu(F_{l-1} - F_l) \leqslant \varepsilon \mu(F)/2^l$, where $F_l = \bigcup_{j=1}^{p_l} F_{j,l}$.
\end{enumerate}
Define $F_* = \bigcap F_l$, we have $\mu(F_*) \geqslant (1-\varepsilon)\mu(F)$. Then choose a subsequence $n_k$ such that $P^{n_k}f$ converges at every $x_{j,l}$. Hence, using the Ascoli-Arzela's arguments yields the uniform convergence on $F_*$ for the family $\{P^{n_k}f\}$.

Now, choose arbitrarily a sequence of compact subsets $\tilde{K}_i$ with $\mu(\tilde{K}_i) \uparrow 1$. By the above result, there exists a subsequence $P^{n_{k,1}}f$ uniformly converging on a compact subset $K_1\subset \tilde{K}_1$ with $\mu(\tilde{K}_1 - K_1) \leqslant 2^{-1}$. Inductively for $i\geqslant 2$, we can always find the $i$-th subsequence $\{P^{n_{k,i}}f\} \subset \{P^{n_{k,i-1}}f\}$ uniformly converging on a compact subset $K_i\subset \tilde{K}_i$ with $\mu(\tilde{K}_i - K_i) \leqslant 2^{-i}$. Hence, we obtain that $P^{n_{k,k}}f$ uniformly converges on each $K_i$ with $\mu(K_i) \uparrow 1$.

Denote by $g$ the limit of $P^{n_{k,k}}f$ on $\bigcup K_i$, and extend it to be $\limsup\limits_{k\to \infty} P^{n_{k,k}}f$ on $X$. Thus $g$ is continuous on $X_\mu$ due to the eventual continuity.

For simplicity, rewrite $n_k$ instead of $n_{k,k}$. The Fatou's lemma gives
   \beq
      && \mu(g) = \int \liminf\limits_{k\to\infty} P^{n_k}f d\mu = \int P^n(\liminf\limits_{k\to\infty} P^{n_k}f) d\mu \\
      && \ \ \ \ \ \ \leqslant \int \liminf\limits_{k\to\infty} P^{n+n_k}f d\mu \leqslant \int \limsup\limits_{k\to\infty} P^{n+n_k}f d\mu \leqslant \int P^ng d\mu = \int g d\mu,
   \eeq
which implies $P^ng = \limsup\limits_{k\to\infty} P^{n+n_k}f = \liminf\limits_{k\to\infty} P^{n+n_k}f$ $\mu$-a.e. on $X_\mu$.
\end{proof}

\begin{lemma} \label{lemEveConAll}
Suppose $P^n$ is eventually continuous on $X_\mu$. Then for any $x\in X_\mu$ and any open set $B$ with $\mu(B)>0$, there exists $k$ with $P^k(x,B)>0$.
\end{lemma}
\begin{proof}
Choose $f\in \mathrm{C_b}$ with $0\leqslant f \leqslant \mathbf{1}_B$ and $\mu(f)>0$, then the ergodicity means $\lim\limits_{n\to \infty} Q^nf(y) = \mu(f)$ for $\mu$-a.e. $y\in X_\mu$. Since $P^n$ is eventually continuous, $\lim\limits_{n\to \infty} Q^nf$ is continuous on $X_\mu$. Hence, there is $k$ with $P^k(x,B)\geqslant P^kf(x) >0$.
\end{proof}

Now let's prove Theorem \ref{thmAsyStaSup}.

\begin{proof}
If $P^n$ is asymptotically stable on $X_\mu$, it is easy to prove Statement (2).

On the contrary, we follow the idea in \cite[Proposition 18.4.3]{MT} with some adjustments. For any $f\in \mathrm{C_b}$ with $|f|\leqslant 1$ and $\mu (f) =0$, according to the monotonicity
  \[ \int |P^nf| d\mu = \int P^m(|P^nf|) d\mu \geqslant \int |P^{m+n}f| d\mu,\]
define $v= \lim\limits_{n\to\infty} \int |P^n f| d\mu$. By Lemma \ref{lemEveConSub}, there is a subsequence $P^{n_k}f$ converging to $g$ on a $\mu$-full set $Y$ and $g$ is continuous on $X_\mu$. Thus we have for all $n\geqslant 1$
  \beqn
     \ \ \ \ \int |g| d\mu = \lim\limits_{k\to\infty} \int |P^{n_k} f| d\mu = v
      = \lim\limits_{k\to\infty} \int |P^{n+n_k} f| d\mu = \int |P^ng| d\mu \label{eqv|g|}
  \eeqn
by the dominated convergence theorem.

Claim that $g$ preserves signs on $Y$. Otherwise, there exist two neighborhoods $O_+$ and $O_-$ with positive $\mu$-mass such that
  \beq
     \{x\in Y: g(x)>0\} \ \ \subset &O_+& \subset\ \ \{x\in X: g(x)>0\},\\
     \{x\in Y: g(x)<0\} \ \ \subset &O_-& \subset\ \ \{x\in X: g(x)<0\}.
  \eeq
Let $z\in X_\mu$ be an aperiodic point. We can find $k_\pm$ by Lemma \ref{lemEveConAll} such that
  \beqn
    P^{k_+}(z,O_+)>0, \ \ \ P^{k_-}(z,O_-)>0. \label{eqO+-}
  \eeqn
The Feller property yields a neighborhood $U$ of $z$ satisfying (\ref{eqO+-}) for all $x\in U$. Then choose a big $l$, writing $l_+ = l - k_+$ and $l_- = l - k_-$, such that $P^{l_{\pm}}(z,U) > 0$ due to the aperiodicity and
  \[ P^l(z,O_{\pm}) \geqslant \int_U P^{k_{\pm}}(y,O_{\pm}) P^{l_{\pm}}(z,dy) >0. \]
Again, the Feller property yields another neighborhood $V$ of $z$ with $P^l(x,O_{\pm})>0$ for all $x\in V$. It follows $|P^lg|< P^l|g|$ on $V$, then $\int |P^lg| d\mu < \int P^l|g| d\mu = \int |g| d\mu$, which contradicts (\ref{eqv|g|}). Hence, the above claim is true.

Consequently, we obtain
  \[ v= \int |g| d\mu = \mu (g) = \mu (f) =0,  \]
which implies by the Fatou's lemma again
   \[ 1 = \lim\limits_{n\to \infty} \int 1-|P^nf| d\mu \geqslant \int 1 - \limsup\limits_{nto \infty}|P^nf| d\mu. \]
Hence, $\lim\limits_{n\to \infty} P^nf \equiv 0$ on $X_\mu$ due to the eventual continuity.
\end{proof}

\bigskip

\section{Global asymptotic stability}
 \label{ASGlo}

In this section, we will prove Theorem \ref{thmAsyStaGloNew} and \ref{thmAsyStaGlo}.

\begin{lemma} \label{lemFellInv}
$X_\mu$ is an invariant set, i.e. $P^n(x,X_\mu) = 1$ for all $x\in X_\mu$ and $n\geqslant 1$.
\end{lemma}
\begin{proof}
Assume $P^n(x,X_\mu) < 1$ for some $x\in X_\mu$. Take $f\in \mathrm{C_b}$ with $0\leqslant f \leqslant \mathbf{1}_{X-X_\mu}$ and $P^nf(x) > 0$. The Feller property yields a neighborhood $O_x$ with $P^nf(y) > 0$ for all $y\in O_x$. It follows $0 = \mu (f) \geqslant \int_{O_x} P^nf(y) d\mu(y) > 0$.
\end{proof}

\subsection{Proof of Theorem \ref{thmAsyStaGloNew}}
The next lemma says that the process will stay in a neighborhood of the ergodic support eventually.

\begin{lemma} \label{lemAsyExitProb}
Under the same conditions and (2) as in Theorem \ref{thmAsyStaGloNew}, for all $x\in X$ and $\varepsilon>0$
  \[ \lim\limits_{n\to \infty} P^n(x,X- X_\mu^\varepsilon) = 0. \]
\end{lemma}
\begin{proof}
Take $\varepsilon'\in (0, \varepsilon)$ and $f\in \mathrm{C_b}$ such that $0\leqslant f \leqslant 1$, $f=1$ on $X- X_\mu^\varepsilon$ and $f=0$ on $X_\mu^{\varepsilon'}$. Set $\gamma = \sup\limits_{x\in X} \limsup\limits_{n\to \infty} P^nf(x)$. Assume $\gamma>0$.

Lemma \ref{lemFellInv} yields $\lim\limits_{n\to \infty} P^nf(x) = 0$ for all $x\in X_\mu$. By the eventual continuity on $X_\mu$, there exists $\delta>0$ such that $\limsup\limits_{n\to \infty} P^nf(x) \leqslant \frac14 \gamma$ for all $x\in X_\mu^{\delta}$. Then for $X_\mu^{\delta}$, Statement (2) gives some $\eta>0$ such that $\limsup\limits_{n\to \infty} P^n(y, X_\mu^\delta) \geqslant \eta$ for all $y\in X$.

Select $x_0\in X$ with $\limsup\limits_{n\to \infty} P^nf(x_0) \geqslant \gamma(1-\frac12\eta)$. The Fatou's lemma gives
   \beq
      &&\limsup\limits_{n\to \infty} P^nf(x_0)
         \ \leqslant\ \int \limsup\limits_{n\to \infty} P^nf(y) dP^m(x_0,y)\\
         &\leqslant& \frac{1}{4}\gamma \cdot P^m(x_0,X_\mu^\delta) + \gamma\cdot P^m(x_0,X - X_\mu^\delta)
         \ =\ \gamma(1-\frac34 P^m(x_0, X_\mu^\delta)).
   \eeq
Then taking the inferior limit in $m$, we have
   \[  \gamma(1-\frac12\eta) \leqslant \gamma(1-\frac34 \limsup\limits_{m\to \infty} P^m(x_0, X_\mu^\delta))
        \leqslant \gamma(1-\frac34 \eta), \]
which contradicts oneself. Hence, $\gamma=0$.
\end{proof}

Now, let's prove Theorem \ref{thmAsyStaGloNew}.

\begin{proof}
If $P^n$ is asymptotically stable on $X$, it is easy to prove Statement (2).

On the contrary, assume Statement (2) is true. Let $f\in \mathrm{C_b}$ with $|f|\leqslant 1$ and $\mu(f) =0$. Then $\lim\limits_{n\to \infty} P^nf(y) = \mu(f) = 0$ for all $y\in X_\mu$. Fix arbitrary $x\in X$, due to the eventual continuity of $P^nf$ on $X_\mu$, for any $\varepsilon > 0$, there exist a compact subset $K\subset X_\mu$ and $\delta>0$ such that
  \[ P(x, X_\mu^\delta - K^\delta) \leqslant \varepsilon, \ \ \textrm{ and } \ \
       \limsup\limits_{n\to\infty} |P^nf(y)| \leqslant \varepsilon, \ \forall y\in K^\delta. \]
Hence, using Lemma \ref{lemAsyExitProb} and Fatou's lemma yields
  \[ \limsup\limits_{n\to \infty} |P^nf(x)| \leqslant \limsup\limits_{n\to \infty} \int_{X_\mu^\delta} |P^{n-1}f(y)| P(x,dy) +  \lim\limits_{n\to \infty} P^n(x,X-X_\mu^\delta) \leqslant 2\varepsilon. \]
The asymptotic stability on $X$ is proved.
\end{proof}

\bigskip

\subsection{Proof of Theorem \ref{thmAsyStaGlo}}
\begin{lemma} \label{lemEnterProb}
Under the same conditions as in Theorem \ref{thmAsyStaGlo}, for all $x\in X$
  \[ \lim\limits_{n\to \infty} P^n(x,X_\mu) = 1. \]
\end{lemma}
\begin{proof}
By $(\mathbf{A}1)$, let $U\subset X_\mu$ be a neighborhood containing $z$. Take $f\in \mathrm{C_b}$ such that $0\leqslant f\leqslant \mathbf{1}_U$ and $P^n(z,U)\geqslant P^nf(z) \geqslant \mu(f)/2 > 0$ for big $n$ due to the weak convergence on $X_\mu$. Then the eventual continuity yields another neighborhood $V$ of $z$ with $\limsup\limits_{n\to \infty} |P^nf(x) - P^nf(z)| \leqslant \mu(f)/4$ for all $x\in V$. So this estimate gives us an increasing sequence of closed subset $V_m$ such that $V_m \uparrow V$ and
   \[ P^n(x,U) \geqslant P^nf(x) \geqslant P^nf(z) - \mu(f)/4 \geqslant \mu(f)/4 \]
for all $n\geqslant m$ and all $x\in V_m$.

On the other hand, $(\mathbf{A}2)$ says for any bounded set $A$ and $V$, there exists $\eta_A>0$ such that for any $x\in A$, there exists $k$ with $P^k(x,V) \geqslant \eta_A$. Choose $m$ (depending on $x$) with $P^k(x,V_m) \geqslant \eta_A/2$, we have due to $U\subset X_\mu$ that for all $n\geqslant m$
  \beqn
    \ \ \ \ \ \ \ P^{n+k}(x,X_\mu) \geqslant P^{n+k}(x,U) \geqslant \int_{V_m} P^n(y,U) P^k(x,dy) \geqslant \mu(f) \eta_A/8 =: \beta_A. \label{eqYUV}
  \eeqn

Using Lemma \ref{lemFellInv}, define an sequence of monotone functions
  \[ \varphi_n(x) := P^n(x,X_\mu) = \int_{X_\mu} P^m(y,X_\mu) P^n(x,dy) \leqslant P^{m+n}(x,X_\mu) = \varphi_{m+n} (x), \]
which implies $\varphi(x) := \lim\limits_{n\to \infty} \varphi_n(x)$ exists. In particular, $\varphi = 1$ on $X_\mu$ and $\varphi \geqslant \beta_A$ on $A$ by (\ref{eqYUV}). The definition gives also $\varphi_{m+n} = P^n \varphi_m$,
which implies $\varphi = P^n\varphi$ by the monotone convergence theorem.

Assume $\varphi(x)<1$ for some $x\notin X_\mu$, we set $\varepsilon = (1-\varphi(x))/2$. By $(\mathbf{A}3)$, there exists a bounded $B$ such that $P^{n_i}(x,B) \geqslant 1-\varepsilon$ for a sequence of $n_i$. Moreover, by the above discussion, $\varphi(y) \geqslant \beta_B$ for all $y\in B$. Then the invariance yields
  \beq
    \varphi(x) \ =\ P^{n_i}\varphi(x)
      &\geqslant& \int_{X_\mu \cup (B-X_\mu)} \varphi(y) P^{n_i}(x,dy) \\
      &\geqslant& P^{n_i}(x,X_\mu) + \beta_B \cdot (P^{n_i}(x,B) - P^{n_i}(x,X_\mu)) \\
      &\geqslant& \varphi_{n_i}(x) + \beta_B \cdot (1-\varepsilon - \varphi_{n_i}(x)) \\
      &\overset{i\to \infty}{\longrightarrow}& \varphi(x) + \beta_B \cdot(1-\varepsilon - \varphi(x)) \ >\ \varphi(x),
  \eeq
which contradicts oneself. Hence, $\varphi\equiv 1$ on the whole $X$.
\end{proof}

Now, let's prove Theorem \ref{thmAsyStaGlo}.

\begin{proof}
Let $f\in \mathrm{C_b}$ with $|f|\leqslant 1$ and $\mu (f) =0$. Then we have for all $x\in X$
  \[ \lim\limits_{n\to \infty} |P^nf(x)| \leqslant \limsup\limits_{n\to \infty} \int_{X_\mu} |P^{n-1}f(y)| P(x,dy) +  \lim\limits_{n\to \infty} P^n(x,X-X_\mu) = 0 \]
by using the weak convergence on $X_\mu$ and Lemma \ref{lemEnterProb}.
\end{proof}

\bigskip

\section{Stochastic 2D Navier-Stokes equations revisited}
\label{2DNS}

To get the unique ergodicity, it suffices to check two assumptions in Theorem \ref{thmFeller2}. Set $X$ to be a Banach space, $z=0$, and $w_t$ the associated stochastic process on $X$ with $w_0 =z$. Hairer and Mattingly \cite[Lemma A.1]{Hairer} gives a prior estimate ($\eta > 0, C > 0$)
   \beqn
     \mathbb{E}\exp(\eta ||w_t||^2) \leqslant C\exp(\eta e^{-\nu t} ||w_0||^2),\ \ \forall t>0,\label{eqHairerPrior}
   \eeqn
which implies by the Chebyshev inequality that for any ball $B(z,R)$ and all time $t$
   \[ P_t(z,B(z,R)^c) \leqslant e^{-\eta R^2}\mathbb{E}\exp(\eta ||w_t||^2 \mathbf{1}_{||w_t|| \geqslant R}) \leqslant Ce^{-\eta R^2},\]
and thus $P_t(z,B(z,R)) \geqslant \frac12$ when $R$ is big.

Combining with E and Mattingly \cite[Lemma 3.1]{Mattingly} that for every $\gamma > 0$ there exists
a time $T_\gamma$ such that
   \beqn
      \inf\limits_{w\in B(z,R)} P^{T_\gamma}(w, B(z,\gamma)) > 0,\label{eqMattingly}
   \eeqn
we have by the semigroup property
   \beq
     \limsup\limits_{t\to \infty} Q^t(z,B(z,\gamma)) &=&
        \limsup\limits_{t\to \infty} \int P^{T_\gamma}(w,B(z,\gamma)) Q^t(z,dw)\\  &\geqslant&
          \limsup\limits_{t\to \infty} \int_{B(z,R)} P^{T_\gamma}(w,B(z,\gamma)) Q^t(z,dw) \ >\ 0.
   \eeq
This gives the lower bound condition $(\mathcal{L}_S)$.

Moreover, the gradient estimate in \cite[Proposition 4.3]{Hairer} reads
   \beqn
    |\nabla P_t\varphi(w)| \leqslant C\exp(\eta ||w_t||^2) (||\varphi||_\infty + e^{-\delta t}||\nabla\varphi||_\infty ), \ \ \forall t>0, \label{eqGradientHairer}
   \eeqn
which implies that $P_t$ is equicontinuous, and thus eventually continuous. Therefore, we get the existence of ergodic measures. The uniqueness follows from the eventual continuity and weak irreducibility as explained in our introduction. Note that, the weak irreducibility still follows from (\ref{eqHairerPrior}) and (\ref{eqMattingly}).

To check the asymptotical stability on $X$, it is sufficient to show that $z$ is aperiodic and Assumption (2) in Theorem \ref{thmAsyStaGloNew} holds, which can both be quickly derived from (\ref{eqHairerPrior}) and (\ref{eqMattingly}) too.

We remark that, (\ref{eqGradientHairer}) is a crucial ingredient in \cite{Hairer}, a very hard and very powerful estimate in the literature of stochastic 2D Navier-Stokes equations. However, if one is concerned only to the unique ergodicity and asymptotic stability, the contraction factor $e^{-\delta t}$ there will not be used.

And we have to admit that, it is indeed more interesting to find some new SPDEs which only hold weak gradient estimates, formally like (\ref{eqGradientWeak}), to exhibit fully the effectiveness of our criteria presented in this paper. But at least, we provide such a possibility to establish the ergodic theory for more complicated stochastic models on infinite dimensional spaces.

\bigskip

\subsection*{Acknowledgements}

{\small It is a great pleasure to thank all the members of our seminar for discussions. In particular, Dr. Yong-Sheng Song proposed firstly the fact in Proposition \ref{propSepa}. The authors thank the financial support from Key Laboratory of Random Complex Structures and Data Science, Academy of Mathematics and Systems Science, Chinese Academy of Sciences (No. 2008DP173182). Respectively, Fu-Zhou Gong is also supported by NSFC (no. 11021161) and 973 Program (no. 2011CB808000), and Yuan Liu supported by NSFC (no. 11201456) and CAS grant (no. Y129161ZZ1).}


\end{document}